\documentclass[12pt]{amsart}
\usepackage[totalwidth=480pt, totalheight=640pt]{geometry}
\usepackage{amsmath, amssymb, amsthm, amsfonts, %mathrsfs,
eucal, enumerate, natbib, layout}
\usepackage[normalem]{ulem}
\usepackage[usenames]{color}
\usepackage{verbatim}
\usepackage{soul}
\usepackage{mathtools}
\usepackage{hyperref}
\hypersetup{colorlinks=true,citecolor={red},linkcolor={blue},urlcolor={violet}}
\usepackage[new]{old-arrows}

\newcommand{\details}[1]{}

\newtheorem{theorem}{Theorem}[section]
\newtheorem*{theorem*}{Theorem}
\newtheorem{corollary}[theorem]{Corollary}
\newtheorem*{corollary*}{Corollary}
\newtheorem{lemma}[theorem]{Lemma}
\newtheorem*{lemma*}{Lemma}

\newtheorem*{claim*}{Claim}

\newtheorem{proposition}[theorem]{Proposition}
\newtheorem*{proposition*}{Proposition}

\newtheorem*{conjecture*}{Conjecture}
\newtheorem{def-proposition}[theorem]{Definition-Proposition}
\theoremstyle{definition}

\newtheorem*{definition*}{Definition}
\newtheorem{remark}[theorem]{Remark}
\newtheorem{notation}[theorem]{Notation}
\newtheorem{example}[theorem]{Example}
\newtheorem*{example*}{Example}
\numberwithin{equation}{section}

\newcommand{\rme}{\mathrm {e}}
\newcommand{\rmi}{\mathrm {i}}

\newcommand{\ZZ}{\mathbb{Z}}
\newcommand{\QQ}{\mathbb{Q}}

\newcommand{\CC}{\mathbb{C}}
\newcommand{\GG}{\mathbb{G}}

\newcommand{\Galmot}{{\mathcal{G}}{\mathrm{al}}_{\mathrm{mot}}}

\newcommand{\rank}{\mathrm{rank}}

\newcommand{\UR}{\mathrm{UR}}

\newcommand{\Lie}{\mathrm{Lie}\,}

\newcommand{\tR}{\widetilde{R}}

\newcommand{\oQQ}{\overline{\QQ}}

\newcommand{\cP}{\mathcal{P}}
\newcommand{\cE}{\mathcal{E}}

\newcommand{\e}{\mathrm{e}}
\newcommand{\ii}{\mathrm{i}}

\begin{document}

\title[Dimension of the motivic Galois group of a 1-motive]
{Dimension of the motivic Galois group of a 1-motive}

\author{Cristiana Bertolin}
\address{Dipartimento di Matematica, Universit\`a di Padova, Via Trieste 63, Padova}
\email{cristiana.bertolin@unipd.it}



\date{\today}

%\commby{}

%%% ----------------------------------------------------------------------

\begin{abstract}

We compute the dimension of the motivic Galois group of a 1-motive $M$ defined over $\CC,$ expressing it explicitly in terms of the rank of the multiplicative group generated by the points defining $M.$ As an application, we obtain a new formulation of the Grothendieck--André period Conjecture in the setting of 1-motives.

\end{abstract}

%%% ----------------------------------------------------------------------

\maketitle

%%% ----------------------------------------------------------------------

%\tableofcontents

\section*{Introduction}

Let $\Omega = \ZZ \omega_1 + \ZZ  \omega_2$ be a lattice in $\CC$ with elliptic invariants $g_2,g_3.$
Let $\cE$ be the elliptic curve associated with $\Omega$ and denote by $k$ its endomorphism field.
Associated with the lattice $\Omega$ are the Weierstrass functions $\wp$, $\zeta$ and $\sigma,$ as well as the Serre function 
$f_{q}(z) = \frac{\sigma(z+q)}{\sigma(z)\sigma(q)} \rme^{-\zeta(q) z },$
 where $q$ is a complex number which does not belong to $\Omega.$

 Consider the 1-motive
\begin{equation} \label{GPC-1-motive}
	M=[u:\ZZ \longrightarrow  G^n],  \quad
	u(1)=\big(R_1, \dots, R_n\big)  \in G^n(\CC),
\end{equation}
where $G$ is the extension of the elliptic curve $\cE$ by $\GG_m^r$ parametrized by the points 
$Q_1 =\exp_{\cE^*}(q_1),   \ldots,  Q_{r} =\exp_{\cE^*}(q_r)$\footnote{In the whole text we use small letters for elliptic logarithms of points of $\cE^*(\CC)$ or $\cE(\CC)$ which are written with capital letters.} of the dual elliptic curve $\cE^*,$ that we identify with $\cE$, and 
\begin{align*}
	\nonumber	R_{i}&=\exp_G(p_i,t_{i1}, \dots ,t_{ir})\\
	\nonumber	&=  \sigma(p_i)^3\Big[ \wp(p_i):\wp'(p_i):1 : \rme^{t_{ij}} f_{q_j}(p_i): \rme^{t_{ij}} f_{q_j}(p_i) \Big( \wp(p_i) + \frac{\wp'(p_i)-\wp'(q_j)}{\wp(p_i)-\wp(q_j)} \Big) \Big]_{j=1, \dots, r}
\end{align*}
for $i=1, \dots,n.$	
 By additivity of the category of extensions, the group variety 
 $G$ decomposes as a product $G_{1}\times \dots\times G_{r},$ where $G_j$ is the extension of $\cE$ by $\GG_m$ parametrized by the point $Q_j$ for $j=1, \dots,r.$	
 Accordingly, each point $R_i $ in the fibre $G_{P_i}$ of $G$ above the point $P_i $ decomposes into $r$ components $R_{ij}$ in the fibre $(G_{j})_{P_i} $ where 
 \begin{equation} \label{R}
 	R_{ij}=\exp_{G_j}(p_i,t_{ij})=  \sigma(p_i)^3\Big[ \wp(p_i):\wp'(p_i):1 : \rme^{t_{ij}} f_{q_j}(p_i): \rme^{t_{ij}} f_{q_j}(p_i) \Big( \wp(p_i) + \frac{\wp'(p_i)-\wp'(q_j)}{\wp(p_i)-\wp(q_j)} \Big) \Big].
 \end{equation}
The 1-motive $M$ is uniquely determined by the $2+r+n+rn$ complex numbers 
\begin{equation}\label{Points}
	g_2 \in \CC, \;\; g_3 \in \CC, \;\; q_j \in \CC \setminus \Omega, \;\;p_i \in \CC \setminus \Omega, \;\; t_{ij} \in \CC
\end{equation}
and hence its field of definition is 
\begin{equation}\label{FieldDefinition}
K:=\QQ \big(g_2,g_3,Q_j,R_{ij }\big)_{ j=1, \dots, r \atop i=1, \dots,n }=\QQ \big(g_2,g_3, \wp(q_j),\wp(p_i),e^{t_{ij}}f_{q_j}(p_i) \big)_{ j=1, \dots, r \atop i=1, \dots,n }.
\end{equation}
 Throughout the paper, we assume that the field of definition $K$ of the 1-motive $M$ is algebraically closed.

Let $\Galmot (M)$ denote the motivic Galois group of the 1-motive $M$. The aim of the present paper is to provide a general formula for the dimension of $\Galmot(M)$ in terms of the rank of the multiplicative subgroup of $G(\CC)$ generated by the points $R_{ij}$ \eqref{R} defining $M.$

The motivic Galois group $\Galmot (M)$ fits into 
the following exact sequence 
\begin{equation}\label{eq:shortexactsequenceUR}
	0 \longrightarrow \UR(M) \longrightarrow \Galmot (M) \longrightarrow \Galmot (\cE) \longrightarrow 0
\end{equation} 
where $\UR(M)$ is its unipotent radical and $\Galmot (\cE)$ is the motivic Galois group of $\cE$, i.e. its maximal reductive quotient (see for example \cite[\S 3.1]{BPSS}).
According to \cite[Théorème 0.1]{B03} 
the Lie algebra of $\UR(M)$ is the semi-abelian variety
\begin{equation}\label{eq:ses}
	0 \longrightarrow Z(1) \longrightarrow  \Lie \UR (M)  \longrightarrow B \longrightarrow 0
\end{equation} 
defined by the adjoint action of the Lie algebra $(B,Z(1),[\cdot,\cdot])$ on $B+Z(1),$ where

-  $B$ is the smallest abelian subvariety of $\cE^n \times \cE^{*s}$ generated by the point $(P_1, \dots, P_n,Q_1,\dots , \\ Q_r)$ modulo isogenies, and 

- $Z(1)$ is the smallest subtorus of $\GG_m^{nr}$ which contains 
the torus $Z'(1)$ defined by
the image of the Lie bracket $[\cdot,\cdot]: B \otimes B \to \GG_m^{nr}$ constructed using the motivic Weil pairing of $\cE$ (see \cite[\S 1.3]{B03}), and the torus $Z(1)/Z'(1)$ defined by the point $ \pi (pr_*\tR)$ constructed in \eqref{piprR} using the points $R_{ij}$ \eqref{R}.

In  previous works \cite{B01,BP} we studied how the geometry of the 1-motive $M$ -- namely, existence of endomorphisms and relations between the points \eqref{Points} -- affects the dimension of $\Galmot (M).$ The short exact sequences \eqref{eq:shortexactsequenceUR} and \eqref{eq:ses} reduce the study of $\Galmot(M)$ to that of the pure motives $B, Z'(1)$ and $Z(1)/Z'(1)$ underlying the Lie algebra $\Lie \UR (M).$

The dimension of the abelian variety $B$ decreases when there are relations among the  $P_i$ and $Q_j$ induced by endomorphisms of $\cE^n \times \cE^{*r}.$ More precisely, the dimension of $B$ is equal to the dimension of the sub $k$--vector space of $\CC / (\Omega \otimes_\ZZ \QQ)$ generated by the classes of the elliptic logarithms $ p_1, \dots, p_n, q_1, \dots,q_r $  modulo $\Omega \otimes_\ZZ \QQ.$

To describe the toric part of $\Lie \UR(M),$ we first decompose $M$ into 1-motives $M_{ij}$ with rank-1 lattice and 1-dimensional torus (see Section \ref{decomposition}). Denote by 
\[	\mathrm{NoLB}\]
the subset of $\{1,\dots,n\} \times \{1,\dots,r\}$ consisting of couples $(i,j)$ such that one of the following conditions is satisfied:
\begin{itemize}
	\item  $P_i$ and $Q_j$ are both torsion,
	\item  $P_i$ or $Q_j$ is a torsion point,
	\item  $P_i$ and $Q_j$ are $k$-linearly dependent via an antisymmetric homomorphism, that is $\phi(P_i)=Q_j$ (or $\phi(Q_j)=P_i$) with $\phi+\overline{\phi}=0.$
\end{itemize}
Moreover set 
\[	\mathrm{LB} := \big(\{1,\dots,n\} \times \{1,\dots,r\} \big) \setminus \mathrm{NoLB}.\]
The notation $\mathrm{LB}$ refers to the pairs contributing to the Lie bracket part of  $\Lie \UR(M),$  whereas $\mathrm{NoLB}$ denotes the pairs for which the Lie bracket vanishes.
 According to \cite[Lemma 3.1 and Corollary 4.5]{BP} the fibre $(G_j)_{P_i}$ is canonically isomorphic to $\cE \times \cE^* \times \GG_m$ if and only if 
 $(i,j) \in \mathrm{NoLB}.$ Equivalently, the fibres corresponding to the pairs in $\mathrm{LB}$ are precisely those which are not canonically
 isomorphic to $\cE \times \cE^* \times \GG_m.$

 Since by \cite[Lemma 3.1]{BP} $Z'(1)$ is the smallest subtorus of $\GG_m^{nr}$ which contains the values of the factor of automorphy of the extension $\Lie \UR (M),$ the dimension of $Z'(1)$ involves only the couples $(i,j) \in \mathrm{LB}.$
Notice, however, that the pairs in LB contribute not only to the
torus \(Z'(1)\), but also to the quotient torus \(Z(1)/Z'(1)\). More
precisely, the Lie bracket determines the subtorus \(Z'(1)\), whereas
the comparison of different trivializations lying over the same
LB-fibre gives rise to an additional contribution to
\(Z(1)/Z'(1)\), as illustrated by Example \ref{example}. By contrast, the pairs in NoLB contribute only to
the quotient torus \(Z(1)/Z'(1)\).

In Notation \ref{Notation} we introduce distinguished points \(R_{i_mj_m},R_{i_aj_a},R_{i_lj_l}\) of the 1-motive $M$ \eqref{GPC-1-motive}, each playing a specific
geometric role. The points \(R_{i_mj_m}\) determine the
fundamental non-split fibres of the 1-motive $M$ and account for the torus \(Z'(1)\). 
Accordingly, in  Proposition \ref{dimZ'(1)} we formulate the dimension of \(Z'(1)\) as the rank of
the multiplicative subgroup of \(G(\CC)\) generated by those
points \(R_{i_mj_m}\).
The
points \(R_{i_aj_a}\) measure the independent differences between
trivializations of these fundamental non-split fibres and account for the LB-contribution
to \(Z(1)/Z'(1)\). In Proposition \ref{dimZ(1)/Z'(1)} we show that the LB-contribution to the quotient torus
\(Z(1)/Z'(1)\) is equal to the rank of the multiplicative subgroup of \(G(\CC)\)
generated by the points \(R_{i_aj_a}\).
Finally, the points \(R_{i_lj_l}\) lie on split
fibres and account for the NoLB-contribution to
\(Z(1)/Z'(1)\). Proposition \ref{dimZ(1)/Z'(1)} also shows that this
NoLB-contribution is equal to the rank of the multiplicative
subgroup of \(G(\CC)\) generated by points \(R_{i_lj_l}.\) Finally, in
Theorem \ref{teo:GeneralPoints} we prove that these distinguished points \(R_{i_mj_m}, R_{i_aj_a}, R_{i_lj_l} \)
generate the same
multiplicative subgroup of \(G(\CC)\) as all the points \(R_{ij}\) \eqref{R}
defining the 1-motive $M$. This yields the following description of the dimension of 
 \(\Galmot(M)\) in terms of the rank of the
multiplicative subgroup generated by the points \(R_{ij}\):

\begin{theorem}\label{Teo:dimGal(M)}
	Let $M=[u:\ZZ \rightarrow  G^n   ],
	u(1)=(R_1, \dots, R_n ) \in G^n(\CC),$ be the 1-motive \eqref{GPC-1-motive} defined by the complex numbers $q_j,p_i,t_{ij}$ \eqref{Points}. Then
	\[	\dim \Galmot(M)	= \frac{4}{\dim_\QQ k} +2 \dim_k \langle p_i,q_j \rangle_{i=1, \dots,n \atop j=1, \dots, r} +
	\mathrm{rank}  \langle  R_{ij}\rangle_{i=1, \dots,n \atop j=1, \dots, r}\]
	where
	\begin{itemize}
		\item $\langle p_i,q_j \rangle_{i,j}$ is the $k $--vector subspace  of  $\CC/(\Omega\otimes_\ZZ\QQ)$ generated by the classes of $p_1, \dots, p_n,q_1, \\ \dots, q_r $ modulo $\Omega\otimes_\ZZ \QQ,$
		\item $ \langle R_{ij}\rangle_{ i,j  }  $ is the multiplicative subgroup of $G(\CC)$ generated by the points $R_{ij} \in G(\CC), i=1, \dots,n$ and $j=1, \dots,r.$
	\end{itemize} 
\end{theorem}

The main idea of the proof is that the rank of the multiplicative subgroup generated by the points \(R_{ij}\) defining $M$ and the dimension of $ \Lie \UR(M)$ grow simultaneously. More precisely, adding a point which is multiplicatively independent from the previously considered ones enlarges the tannakian category and therefore contributes a new dimension to $\Lie \UR(M)$. Conversely, if a point does not increase the dimension of $\Lie \UR(M)$, then it is multiplicatively dependent on the previously chosen points.

As an application we derive a new formulation of the Grothendieck-André period Conjecture for 1-motives stated in \cite[Corollary 6.6]{Bsubmitted}.

\begin{corollary} 
	Let  $
	M=[u:\ZZ \rightarrow G^n ], u(1) =(R_1, \dots, R_n ) ,$ be the 1-motive \eqref{GPC-1-motive} defined by the complex numbers $p_i, q_j$ and $t_{ij}$ \eqref{Points}. Without loss of generality, we assume that the classes of the elliptic logarithms $p_i , i=1, \dots, n',$ and  $q_j ,j=1, \dots, r',$ form a $k$--basis of the sub $k $--vector space $\langle p_i,q_j \rangle_{i,j}$ of  $\CC/(\Omega\otimes_\ZZ\QQ)$ generated by the classes of $p_1, \dots, p_n,q_1, \dots, q_r .$ 
	 Let 
	 
	 - $\{R_{i_h,j_h}\}_{h=1, \dots, \mathrm{rank}  \langle  R_{ij}\rangle_{  (i,j) \in \mathrm{LB} }}$ be generators of $\langle R_{ij}\rangle_{(i,j)\in \mathrm{LB}}\subseteq G(\mathbb{C})\footnote{By \cite[Corollary 4.5]{BP}, for $i \leqslant n'$ and $j \leqslant r'$ the couple $(i,j) \in \mathrm{LB}$. Moreover, according to \cite[Lemma 6.4]{Bsubmitted}, $ \mathrm{rank}  \langle  R_{ij}\rangle_{  (i,j) \in \mathrm{LB} } \geqslant n'r'$ and so, without loss of generality, we may assume $ i_h =1, \dots, n'$ and  $ j_h =1, \dots, r'$ for $h=1, \dots,n'r'.$} $, and 
	 
	 - $\{\mathrm{e}^{t_{i_l j_l}}\}_{l=1, \dots, \mathrm{rank} \langle \rme^{ t_{ij}}\rangle_{ (i,j) \in \mathrm{NoLB}}}$ be generators of $\langle \mathrm{e}^{t_{ij}}\rangle_{(i,j)\in \mathrm{NoLB}}\subseteq \mathbb{G}_m(\mathbb{C}).$

	\par\noindent The Gro\-then\-dieck-Andr\'{e} period Conjecture applied to $M$ reads 
	\[ \mathrm{t.d.}\,\QQ\big(  g_2,g_3,  \wp(q_j),\wp(p_i), \rme^{ t_{i_hj_h}} f_{q_{j_h}}(p_{i_h}) ,\rme^{ t_{i_lj_l}} , \omega_{1}, \eta_{1},\omega_{2}, \eta_{2},  p_i  ,  \zeta(p_i),  q_j,  \zeta(q_j), t_{i_hj_h}   ,t_{i_lj_l}  \big)_{i,j,h,l}\]
	\[\geqslant	 \frac{4}{\dim_\QQ k} +2 (n'+r')+ \mathrm{rank}  \langle  R_{ij}\rangle_{  (i,j) \in \mathrm{LB} } + \mathrm{rank} \langle \rme^{ t_{ij}}\rangle_{ (i,j) \in \mathrm{NoLB}}.\]	
	If $\QQ(g_2,g_3,Q_j,R_i)_{i,j} \subseteq \oQQ,$ then equality holds.
\end{corollary}

%---------------------------------------------------------
\section*{aknowledgement}
We thank Michel Waldschmidt for his remarks, which greatly improved this paper.

%--------------------------------------------------------------
\section{Notation}
Let $G$ be an extension of an elliptic curve $\cE$ by the multiplicative group $\GG_m.$
A point $R=\exp_G(p, t) \in G(\CC)$ is torsion if there exists a nonzero $a \in \mathbb{Z}$ such that $R^a = 1.$ This is equivalent to the conditions
\begin{align*}
	ap & =\omega \\
	at &  = 2 \pi \rmi s
\end{align*}
for some $\omega \in \Omega$ and $s \in \ZZ.$
Points $R_1, \dots, R_n$ in $G(\CC)$ are multiplicatively dependent if there exists an $n$-tuple $(a_1, \dots, a_n) \in \mathbb{Z}^n \setminus \{(0, \dots, 0)\}$ such that
\[
R_1^{a_1} \cdots R_n^{a_n} = 1.
\]
They are said to be multiplicatively independent if for every \((a_1, \dots, a_n) \in \mathbb{Z}^n\) such that
$R_1^{a_1} \cdots R_n^{a_n} = 1,$
one necessarily has $(a_1, \dots, a_n) = (0, \dots, 0).$
The multiplicative subgroup of $G (\CC)$ generated by $n$-points $R_1, \dots,R_n$ of $G(\CC)$ is
\[\langle R_1,\dots,R_n \rangle
=
\left\{\, R_1^{a_1}\cdots R_n^{a_n} \;\middle|\; (a_1,\dots,a_n)\in \mathbb{Z}^n \,\right\}.\]
Its rank is the maximal number of multiplicatively independent elements it contains.

Let $R = \exp_G(p, t)$ and $R' = \exp_G(p', t')$ be two points of $G(\CC).$ Then the equality $R'^b = R^a$ for some $(a,b) \in \mathbb{Z}^2 \setminus \{(0,0)\}$ holds if and only if
\begin{align}\label{def:R'in<R>}
	\nonumber	bp' &  = ap + \omega \\
	bt' &  = at + 2 i \pi s
\end{align}
for some $\omega \in \Omega$ and $s \in \mathbb{Z}.$ In particular if $b=1$, this shows that $R' \in \langle R \rangle.$ More generally, if $\omega \in \Omega \otimes \mathbb{Q}$ and $s \in \mathbb{Q}$, the above conditions \eqref{def:R'in<R>} are equivalent to requiring that $R'^b R^{-a}$ is a torsion point of $G(\CC).$

%-------------------------------------------------------------------

 %-------------------------------------------------
 
 \section{Decomposition into 1-motives with rank-1 lattice and 1-dimensional torus}\label{decomposition}
 
 	Let $M=[u:\ZZ \rightarrow  G^n   ],
 u(1)=(R_1, \dots, R_n ) \in G^n(\CC),$ be the 1-motive \eqref{GPC-1-motive} defined by the complex numbers $q_j,p_i,t_{ij}$ \eqref{Points}. In order to describe the Lie algebra of the unipotent radical $\UR(M)$ of $M,$ we proceed in three steps:

\begin{enumerate}
	\item Let 
	\[B\]
	be 
	the smallest abelian sub-variety (modulo isogenies) of $\cE^n\times \cE^{*r}$
	which contains a multiple of the point $(P_1,\dots,P_n,Q_1,\dots,Q_r) \in \cE^n \times \cE^{*r} (\CC).$
	
	\medspace
	
	\item  Using the Poincar\'e biextension $\mathcal{P}$ of $(\cE,\cE^*)$ by $\GG_m,$
	in \cite[Example 2.8]{B03} we have constructed explicitly a biextension $\mathcal{B}$ of $(\cE^n \times \cE^{*r},\cE^n \times \cE^{*r})$ by $\GG_m^{nr},$ whose pull-back $d^* \mathcal{B}$ via the diagonal morphism $d:\cE^n \times \cE^{*r} \to (\cE^n \times \cE^{*r}) \times (\cE^n \times \cE^{*r})$ is a $\GG_m^{nr}$-torsor over $\cE^n \times \cE^{*r}$ inducing a Lie bracket 
	$ [\cdot,\cdot]: (\cE^n \times \cE^{*r}) \otimes (\cE^n \times \cE^{*r}) \to \GG_m^{nr}$
	(see \cite[Lemma 3.3, p.600]{B03} and see \cite[(2.8.4)]{B03} for an explicit description of this Lie bracket).
	The pull-back 
	$I^*d^* \mathcal{B}$ of  the $\GG_m^{nr}$-torsor $d^* \mathcal{B}$ via the inclusion $I: B \hookrightarrow \cE^n \times \cE^{*r}$ induces the restriction of Lie bracket to $B$, that is $[\cdot,\cdot]: B \otimes B \to \GG_m^{nr}.$ Let
	\[Z'(1)\]
	be the smallest sub-torus of $\GG_m^{nr}$ which contains the image of the Lie bracket $[\cdot,\cdot]: B \otimes B \to \GG_m^{nr}.$ 
	In \cite[Lemma 3.1]{BP} we have showed that $Z'(1 )$ coincides with the smallest sub-torus of $\GG_m^{nr}$ which contains the values of the factor of automorphy of the  $\GG_m^{nr}$-torsor $I^*d^*\mathcal{B}$. In particular,  the push-down ${pr}_*I^*d^* \mathcal{B}$ via the projection $pr:\GG_m^{nr} \twoheadrightarrow \GG_m^{nr} / Z'(1)$ of the torsor $I^*d^* \mathcal{B}$
	is \textit{the trivial $\GG_m^{nr}  / Z'(1)$-torsor over $B$}, i.e. 
	${pr}_*I^*d^* \mathcal{B}= B \times \GG_m^{nr} / Z'(1) .$
	
	\medspace

	\item Let $ v:\ZZ^n\to \cE$ and $v^*:\ZZ^r\to \cE^* $ 
	be the group homomorphisms determined by $v(x_i)=P_i$ and $v^*(y_j^\vee)=Q_j,$
	where $x_1,\ldots,x_n$ and $y_1^\vee,\ldots,y_r^\vee$ denote the standard bases of
	\(\ZZ^n\) and \(\ZZ^r\), respectively. Let $G_j$ be the extension of the elliptic curve $\cE$ by $\GG_m$ parametrized by the point $Q_j=\exp_{\cE^*}(q_j).$ 
	By \cite[\S 1.2]{BP}, having the group homomorphism  $u: \ZZ \to G^n, u(1)=(R_{ij})_{i,j}$ is equivalent to having
	a trivialization (= biadditive section) $\psi : \ZZ \times \ZZ \longrightarrow  (v \times v^*)^* \cP$ of the pull-back $(v \times v^*)^* \cP$ via $v \times v^*$ of the Poincar\'e biextension $\mathcal{P}$ of $(\cE,\cE^*)$ by $\GG_m$:
	\[R_{ij} = \psi(x_i,y^\vee_j) \in \cP_{P_i,Q_j} \simeq (G_j)_{P_i}. \]

	\par\noindent Consider the two group homomorphisms $V: \ZZ \to \cE^n$ and $V^*: \ZZ \to \cE^{*r} $  defined by the points $P_1,\dots,P_n$ in $ \cE^n(\CC)$ and $Q_1,\dots,Q_r $ in $ \cE^{*r}(\CC)$ respectively. According to \cite[\S 3.1]{BP}, having
	the trivialization $\psi : \ZZ \times \ZZ \longrightarrow  (v \times v^*)^* \cP$ is equivalent to having
	a trivialization  $\Psi : \ZZ \times \ZZ \longrightarrow  (V \times V^*)^* I^*d^* \mathcal{B}$ of the pull-back $(V \times V^*)^*I^*d^* \mathcal{B}$ via $V \times V^*$ of the $\GG_m^{nr}$-torsor $I^*d^* \mathcal{B}$ over $B$.
	The trivialization $\Psi$ defines  
	a point 
	$	\Psi(1,1) = \big(\psi(x_i,y^\vee_j)\big)_{i,j }  \in   \big( (V \times V^*)^* I^*d^*\mathcal{B}\big)_{1,1} $
	which in turn furnishes a point
	\[\tR \in (I^*d^*\mathcal{B})_{(P,Q)}\]
	in the fibre of $I^*d^*\mathcal{B}$ over the point $(P,Q)=(P_1,\dots,P_n,Q_1,\dots,Q_r)  \in  B .$ 
	Because of the equality
	$ (V \times V^*)^*  {pr}_*I^*d^* \mathcal{B}  =  {pr}_*  (V \times V^*)^* I^*d^* \mathcal{B} ,$
	the  trivialization $\Psi$ defines a trivialization 
	$
	{pr}_*\Psi : \ZZ \times \ZZ \longrightarrow  (V \times V^*)^*  {pr}_* I^*d^* \mathcal{B} 
	$
	of the pull-back via $ V \times V^*$ of the trivial torsor $  {pr}_*I^*d^* \mathcal{B}$. Denote by $\pi: {pr}_*I^*d^* \mathcal{B} \twoheadrightarrow \GG_m^{nr}/Z'(1)$ the projection on the second factor.
	The point 
	$	{pr}_*\Psi(1,1) \in  \big(  (V \times V^*)^* {pr}_* I^*d^*\mathcal{B}\big)_{1,1} $
	corresponds to the point 
	\[	pr_* \tR =\big((P,Q), \pi (pr_* \tR)\big) =  \big((P,Q), \pi (pr_*  (\psi(x_i,y^\vee_j))_{i= 1, \dots, n \atop j = 1, \dots, r } )\big).
	\]
	in the fibre of the trivial torsor $pr_* I^*d^*\mathcal{B} =B\times \GG_m^{nr}/Z'(1)$ over the point $(P,Q) \in  B.$ Let
	\[Z(1)\]
	be the smallest sub-torus of $\GG_m^{nr}$ which contains $Z'(1)$ and such that $Z(1)/Z'(1)$ contains the point 
	\begin{equation}\label{piprR}
		\pi (pr_* \tR) =\pi \big(pr_*  \big(\psi(x_i,y^\vee_j)\big)_{i= 1, \dots, n \atop j = 1, \dots, r } \big) .
	\end{equation}
	Notice that the point $	\pi(pr_*\widetilde R)$
	depends on the trivialization $\psi$ and therefore on the points $R_{ij}$ defining the $1$-motive.
\end{enumerate}

Consider the inclusion 
$I:B \hookrightarrow \cE^{n}\times  \cE^{*r}, b \mapsto \big(\gamma_1(b),\dots,\gamma_{n}(b),\gamma_{1}^*(b),\dots ,\gamma_{r}^*(b)\big),$
where  $\gamma_i\in\mathrm{Hom}_{\mathbb Q}(B,\cE) $ ({\it resp.} $\gamma^*_j \in\mathrm{Hom}_{\mathbb Q}(B,\cE^*)$) is the composition of $I$ with the projection on the $i$-th factor of $\cE^{n}$ ({\it resp.} on the $j$-th factor of $\cE^{*r}$) for $i=1,\dots,n$ ({\it resp.} $j=1,\dots,r$).
Set $\beta_{i,j}: =\gamma_i^t\circ\gamma_j^* \in \mathrm{Hom}_{\mathbb Q}(B,B^*),$ where the upper-index ${}^t$ denotes the transpose of a group morphism. 
 By \cite[Theorem 4.2]{BP} the kernel of the surjective map
\begin{equation}
	\begin{matrix}
		f:( \ZZ^{n} \otimes \ZZ^{r}) \otimes_\ZZ \QQ &  \twoheadrightarrow &\sum_{i,j}\QQ (\beta_{i,j} + \beta_{i,j}^t) \subset \mathrm{Hom}_\QQ (B,B^*)\\
		\hfill x_i\otimes y_j^\vee &\mapsto & \beta_{i,j} + \beta_{i,j}^t \hfill
	\end{matrix}
\end{equation}
is the space of all $\QQ$--linear relations among the homomorphisms
$\beta_{ij}+\beta_{ij}^t$.

By \cite[Theorem 5.1]{BP} we have that 
\begin{align}\label{dimB-Z'-Z/Z'}
\nonumber	\dim B &= \dim_k \langle p_i,q_j\rangle_{i,j}\\
	\dim Z'(1) &= \dim_\QQ \langle \beta_{i,j}+\beta_{i,j}^t\rangle_{  (i,j) \in \mathrm{LB}}\\
\nonumber	\dim Z(1)/Z'(1) &= \dim_{\QQ}
	\langle \textstyle{\sum_{(i,j)\in \mathrm{LB}}}	\alpha_{ij}\log( s_{ij}) \rangle
	+\dim_\QQ \langle t_{ij}\rangle_{ (i,j) \in \mathrm{NoLB}}
\end{align}

where 
\begin{itemize}
	\item  $\langle p_i,q_j\rangle_{i,j}$ is the sub $k$--vector space of $\CC / (\Omega \otimes_\ZZ \QQ)$ generated by the classes of the complex numbers $ p_1, \dots, p_n, q_1, \dots,q_r $  modulo $\Omega \otimes_\ZZ \QQ,$
	\item $\langle \beta_{i,j}+\beta_{i,j}^t \rangle_{  (i,j) \in \mathrm{LB}}$ is the sub $\QQ$--vector space of $\mathrm{Hom}_\QQ(B,B^*)  := \mathrm{Hom}(B,B^*) \otimes_\ZZ \QQ $ generated by the group homomorphisms $\beta_{i,j}+\beta_{i,j}^t$ with $(i,j) \in  \mathrm{LB},$
	\item $	\langle \textstyle{\sum_{(i,j)\in \mathrm{LB}}}	\alpha_{ij}\log( s_{ij}) \rangle$ is the sub $\QQ$--vector subspace of  $\CC/ 2 \pi \ii \QQ$ generated by the classes of the logarithms  $\sum_{i,j}\alpha_{ij}\log(s_{ij})$, with $\sum_{i,j}\alpha_{ij} x_i\otimes y_j^\vee \in \ker (f)$ and $(s_{ij})$ any point of $\GG_m^{nr}$ projecting onto $ \pi \big( pr_*\big(\psi(x_i, y_j^\vee)\big)_{(i,j)\in \mathrm{LB}} \big),$
	\item  $\langle t_{ij} \rangle_{ (i,j) \in \mathrm{NoLB}}$ is the sub $\QQ$--vector space of $\CC / 2 \pi \ii \QQ$ generated by the classes of the complex numbers $t_{ij} $  modulo $2 \pi \ii \QQ$ with $(i,j) \in  \mathrm{NoLB}$.
\end{itemize}

\begin{remark}\label{Contribution}
The dimension of the quotient torus $Z(1)/Z'(1)$
splits naturally into a LB- and a NoLB-con\-tri\-bu\-tion. The
former measures the independent differences between trivializations
lying in the same LB-fibres (see Example \ref{example}), while the latter comes from the
NoLB-fibres (that is, split fibres), whose toric coordinates survive unchanged after
quotienting by $Z'(1).$ Since these contributions arise from disjoint
families of fibres, they are independent and their dimensions add.
\end{remark}

The third equality of \eqref{dimB-Z'-Z/Z'} reflects a dichotomy satisfied by the elementary
$1$-motives
\[
M_{ij}=[u_{ij}:\ZZ\longrightarrow G_j],\qquad
u_{ij}(1)=R_{ij}
\]
for $i=1, \dots n$ and $j=1, \dots,r.$ Add the index $i,j$ to the pure motives underlying the Lie algebra of the unipotent radical of the 1-motive $M_{ij}: B_{ij} \subseteq \cE\times \cE^*, Z'_{ij}(1) \subseteq \GG_m, Z_{ij}(1) \subseteq \GG_m.$	For every pair $(i,j)$, exactly one of the following two situations occurs:

\begin{enumerate}
	\item[(i)] $(i,j)\in\mathrm{LB}$. Then
	\[
	\dim Z'_{ij}(1)=1,
	\qquad
	\dim Z_{ij}(1)/Z'_{ij}(1)=0.
	\]
	
	\item[(ii)] $(i,j)\in\mathrm{NoLB}$. Then
	\[
	\dim Z'_{ij}(1)=0, \qquad
	\dim Z_{ij}(1)/Z'_{ij}(1)=1 \Leftrightarrow t_{ij} \notin 2\pi \rmi \QQ .
	\]
\end{enumerate}

According to \cite[Lemma 2.2]{B19}, $M$ and $\oplus_{j=1}^r\oplus_{i=1}^n M_{ij}$ generate the same tannakian category and so they have the same motivic Galois group. We therefore obtain the inequality
\[
\dim \Galmot (M) = \dim \Galmot (\oplus_{j=1}^r\oplus_{i=1}^n M_{ij})  \leqslant \oplus_{j=1}^r \oplus_{i=1}^n \dim \Galmot (M_{ij})
\]
and in particular
\begin{equation}\label{reduction}
	\dim \UR (M) = \dim \UR (\oplus_{j=1}^r\oplus_{i=1}^n M_{ij})  \leqslant \oplus_{j=1}^r \oplus_{i=1}^n \dim \UR(M_{ij}).
\end{equation}
The inequality \eqref{reduction} holds only for the pure motives $B,Z'(1)$ and $Z(1),$ that is 
\begin{align}\label{reductionPureMotifs}
	\nonumber	 \dim B & \leqslant \oplus_{j=1}^r \oplus_{i=1}^n  \dim B_{ij},\\
	\dim Z'(1) & \leqslant \oplus_{j=1}^r \oplus_{i=1}^n \dim Z'_{ij}(1),\\
	\nonumber	 \dim Z(1) & \leqslant \oplus_{j=1}^r \oplus_{i=1}^n \dim Z_{ij}(1).
\end{align}
The analogous inequality for the torus $ Z(1)/Z'(1) $
is, in general, false. Indeed, the Lie Bracket contribution to
\(Z(1)/Z'(1)\) is a global phenomenon which cannot be detected on the
individual $1$-motives $M_{ij}$, as illustrated by the following example.

\begin{example}\label{example}
	 Let $p,q$ be two $k$--linearly independent elliptic logarithms and 
let $G$ be the extension of the elliptic curve $\cE$ by the multiplicative group $\GG_m$ parametrized by the point $Q$. 
Consider the three 1-motives  
\begin{align*}
	M_{R}&=[u:\ZZ \rightarrow  G  ], u(1)= R =\exp_G(p, t), \mathrm{with} \; t \in \CC\\
	M_{R'}&=[u':\ZZ \rightarrow  G  ],	u'(1)= R' =\exp_G( p, t'), \mathrm{with} \;  t' \in \CC,\\ 
	M_{R,R'}&= M_R \oplus M_{R'}.
\end{align*}
We add the index $R$ (\emph{resp.} $R'$ and $R,R'$) to the pure motives underlying the Lie algebra of the unipotent radical of the 1-motive $M_R$ (\emph{resp.} $M_{R'}$ and $M_{R,R'}$). Clearly the abelian varieties $ B_{R,R'}, B_R$ and $  B_{R'}$ coincide.
The tori $Z'_{R,R'}(1), Z'_R(1)$ and $Z'_{R'}(1) $ have dimension 1 (a generator is $\beta_{p,q}+\beta^t_{p,q}$) and 
 $\dim Z_R(1) / Z'_R(1)= \dim Z_{R'}(1) / Z'_{R'}(1)=0$ (see \cite[Corollary 4.5]{BP}).
 Both points $R$ and $R'$
 determine the same generator $\beta_{p,q}+\beta^t_{p,q}$
 of the torus $Z'_{R,R'}(1).$ Hence their difference is
 invisible in $Z'_{R,R'}(1)$ and appears only in the quotient $Z_{R,R'}(1)/Z'_{R,R'}(1).$
  It is measured by the class of $t-t'$
 modulo $2 \pi \ii\QQ.$ Hence
 \[
 \dim Z_{R,R'}(1)/Z'_{R,R'}(1)=1 \quad \Longleftrightarrow \quad t-t'\notin 2\pi\ii\QQ .
 \]
 The additional dimension comes
 from comparing two distinct trivializations of the same fibre. Equivalently,
	$$  \dim Z_{R,R'}(1)= 
	\begin{cases}
		1 &  \text{if} \quad  R' \in \langle R \rangle  , \\
		2 & \text{if} \quad R' \notin \langle R \rangle . 
	\end{cases}
	$$
\end{example}

By \cite[Corollary 4.6]{BP}, if the classes of the elliptic logarithms $p_1, \dots, p_n, q_1, \dots, q_r$ modulo $\Omega \otimes_\ZZ \QQ$ are $k$--linearly independent, then
\[	\dim \Galmot(M) = \frac{4}{\dim_\QQ k} +2 (n+r)+nr. \]
To our knowledge, this is the only explicit computation of this dimension in the non-split case (if the extension $G$ is split see Remark \ref{RemarkSplit}). We now express this dimension in terms of the rank of the multiplicative group generated by the points $R_{ij}.$

\begin{lemma} \label{RankForLI} 
	Let $M=[u:\ZZ \rightarrow  G^n   ],
	u(1)=(R_1, \dots, R_n ) \in G^n(\CC),$ be the 1-motive \eqref{GPC-1-motive} defined by the complex numbers $q_j,p_i,t_{ij}$ \eqref{Points} such that the classes of $q_1, \dots, q_r, p_1, \dots, p_n$ modulo $\Omega \otimes_\ZZ \QQ$ are $k$--linearly independent. Assume $r, n \not=0$. The points $R_{ij} \in G(\CC), i=1, \dots, n, j=1, \dots,r,$ are multiplicatively independent.
	
	In particular,$$\dim \Galmot(M)	= \frac{4}{\dim_\QQ k} +2 (n+r)+\mathrm{rank}  \langle R_{ij}\rangle_{  i=1, \dots, n \atop j=1, \dots,r } .$$
\end{lemma}

\begin{proof}
	 Assume by contradiction that there exists a nontrivial relation
	\[
	\prod_{(i,j)} R_{ij}^{a_{ij}} = 1,
	\qquad a_{ij} \in \mathbb{Z},
	\]
	with not all $a_{ij}$ equal to zero.
	Fix $j \in \{1,\dots,r\}$. Since the extension $G$ is the product $G_1 \times \dots \times G_r$ of the extensions $G_j ,$ this relation decomposes componentwise in each $G_j$, giving the equality
	$
	\prod_{i=1}^{n} R_{ij}^{a_{ij}} = 1 $ in $ G_j.
	$
	Applying  the natural projection $\Pi_j : G_j \to \cE$ to this equality, we obtain
	$
	\sum_{i=1}^{n} a_{ij} P_i = 0 $ in $ \cE.
	$
	By assumption, the points $P_1, \dots, P_{n}$ are $k$--linearly independent and so
	$
	a_{1j} = \cdots = a_{nj} = 0.
	$
	Since this holds for every $j = 1, \dots, r$, it follows that all coefficients $a_{ij}$ are zero, contradicting our assumption. Hence the $nr$ points $R_{ij}$ are multiplicatively independent. For the last statement recall that
$\dim \Galmot(\cE)  = \frac{4}{\dim_\QQ k}$ and by \cite[Corollary 4.6]{BP} $\dim Z(1) = \dim Z'(1)=nr.$ Hence using the short exact sequence \eqref{eq:shortexactsequenceUR} we conclude.
\end{proof}

\begin{remark}\label{RemarkSplit}
The cases $n=0$ and $r=0$ are dual to each other through the Cartier duality for 1-motives. The case $r=0$ (equivalently, $q_j \in \Omega \otimes_\mathbb{Z} \mathbb{Q}$ for all $j$) is treated in \cite[Lemma 5.2 and Proposition 5.4]{B02} where we prove that if $M=[u: \ZZ \to \cE^n \times \GG_m^r], u(1)=(P_i,\rme^{t_j})_{i=1, \dots, n,\atop j=1, \dots, r},$ then 
$\dim B =\;  \dim_k \langle p_i\rangle_{i=1, \dots,n },	\dim Z'(1)  =0$ and 
	$\dim Z(1) = \dim Z(1) /Z'(1) = \mathrm{rank}  \langle \rme^{t_j}\rangle_{j=1, \dots,r }.$
	In particular 
	\begin{equation}\label{DimSplit}
			\dim \Galmot(M)	= \frac{4}{\dim_\QQ k} +2  \; \dim_k \langle p_i\rangle_{i=1, \dots,n }+\mathrm{rank}  \langle \rme^{t_j}\rangle_{j=1, \dots,r }.
			\end{equation}
\end{remark}

We finish this section computing the dimensions of the pure motives underlying $\UR(M_{ij})$ in terms of the complex numbers $q_j,p_i,t_{ij}$ \eqref{Points} defining the 1-motive $M_{ij}:$

\begin{description}
	\item[(a)] $P_i$ and $Q_j$ are torsion. Modulo isogenies we assume $P_i=Q_j=0$ and so $M_{ij}=[u_{ij}:\ZZ \to \GG_m \times \cE]$ with
	$u_{ij}(1) =(0,\e^{t_{ij}}) \in \cE \times \GG_m (\CC)$. In this case $\dim B_{ij}= \dim Z'_{ij}(1)=0$ and $\dim Z_{ij}(1) / Z'_{ij}(1)=1$ if and only if  $t_{ij} \notin 2 \pi \ii \QQ.$ 	
	
	 This 1-motive  $M_{ij}=[u_{ij}:\ZZ \to \GG_m \times \cE]$ with
	 $u_{ij}(1) =(0,\e^{t_{ij}}) $ generates the same tannakian category as the 1-motive $[0 \to \cE] \oplus[u_{ij}':\ZZ \to \GG_m]  $ with $u_{ij}'(1) =\e^{t_{ij}} \in \GG_m (\CC).$
	\item[(b)] $Q_j$ is torsion but not $P_i$. Modulo isogenies we assume $Q_j=0$ and so $M_{ij}=[u_{ij}:\ZZ \to \GG_m \times \cE]$ with
	$u_{ij}(1) =(P_i,\e^{t_{ij}}) \in \cE \times \GG_m (\CC).$ In this case $\dim B_{ij}= 1, \dim Z'_{ij}(1)=0$ and $\dim Z_{ij}(1) / Z'_{ij}(1)=1$ if and only if  $t_{ij} \notin 2 \pi \ii \QQ.$
	
	 This 1-motive  $M_{ij}=[u_{ij}:\ZZ \to \GG_m \times \cE]$ with
	$u_{ij}(1) =(P_i,\e^{t_{ij}}) $ generates the same tannakian category as the 1-motive $[u_{ij}': \ZZ \to \cE] \oplus[u_{ij}'':\ZZ \to \GG_m]  $ with  $u_{ij}'(1) =P_i \in \cE (\CC)$ and $u_{ij}''(1) =\e^{t_{ij}} \in \GG_m (\CC).$
	
	 	\item[(c)] $P_i$ is torsion but not $Q_j$. Modulo isogenies we assume $P_i=0$ and so $M_{ij}=[u_{ij}:\ZZ \to G_j]$ with
	 $u_{ij}(1) = R_{ij}= \exp_{G_j} (0,t_{ij}) \in G(\CC).$ In this case $\dim B_{ij}= 1, \dim Z'_{ij}(1)=0$ and $\dim Z_{ij}(1) / Z'_{ij}(1)=1$ if and only if  $t_{ij} \notin 2 \pi \ii \QQ.$ In particular the homomorphism $u_{ij}:\ZZ \to G_j$ factorizes via the torus  $\GG_m$, that is $u_{ij}:\ZZ \to \GG_m \hookrightarrow G_j,$ and if $\Pi : G_j \to \cE$ is the natural projection, $\Pi(R_{ij})=0.$ 
	 
	 This 1-motive $M_{ij}=[u_{ij}:\ZZ \to G_j]$ with
	 $u_{ij}(1) = R_{ij}= \exp_{G_j} (0,t_{ij})$ generates the same tannakian category as the 1-motive $[0 \to G] \oplus[u_{ij}':\ZZ \to \GG_m]  $ with $u_{ij}'(1) =\e^{t_{ij}} \in \GG_m (\CC).$
	 \item[(d)] $P_i$ and $Q_j$ are $k$-linearly dependent. We distinguish two cases:

	 \begin{description}
	 	\item[(d.1)] $\phi(P_i)=Q_j$ (or $\phi(Q_j)=P_i$) with $\phi$ an antisymmetric homomorphism. Since $ \dim Z'_{ij}(1)=0,$  
	 	 the restriction ${\mathcal P}_{\vert B_{ij} } $ is trivial or of order two in $\mathrm{Pic}(B).$ We have
	 $M_{ij}=[u_{ij}:\ZZ \to G_j]$ with
	 	$u_{ij}(1) = R_{ij} $ defined by the point $ (P_i,Q_j,t_{ij}) \in  (G_{j})_{P_i} \cong \cP_{P_i,Q_j} =\{P_i,Q_j\} \times \GG_m $ (or $ 2R_{ij} $ is defined by the point $ (2P_i,Q_j,2t_{ij}) \in \cP_{P_i,Q_j}^2 =\{2P_i,Q_j\} \times \GG_m $). In this case $\dim B_{ij}= 1$ and $\dim Z_{ij}(1) / Z'_{ij}(1)=1$ if and only if  $t_{ij} \notin 2 \pi \ii \QQ.$
	 	 If $\Pi : G_j \to \cE$ is the natural projection, $\Pi(R_{ij})=P_i.$
	 
	 	\item[(d.2)] $\phi(P_i)=Q_j$ (or $\phi(Q_j)=P_i$) with $\phi$ a non antisymmetric homomorphism. We have $M_{ij}=[u_{ij}:\ZZ \to G_j]$ with
	 	$u_{ij}(1) = R_{ij}= \exp_{G_j} (p_i,t_{ij}) \in G(\CC)$. In this case $\dim B_{ij}= 1, \dim Z'_{ij}(1)=1$ and
	 	 $\dim Z_{ij}(1) / Z'_{ij}(1)=0$ (remark that here we have the equality $\dim Z_{ij}(1) / Z'_{ij}(1)=0$ independently of the complex number $t_{ij}$).
	 		 \end{description}

	 \item[(e)] $P_i$ and $Q_j$ are $k$-linearly independent. We have $M_{ij}=[u_{ij}:\ZZ \to G_j]$ with
	 $u_{ij}(1) = R_{ij}= \exp_{G_j} (p_i,t_{ij}).$ In this case $\dim B_{ij}= 2, \dim Z'_{ij}(1)=1$ and $\dim Z_{ij}(1) / Z'_{ij}(1)=0$  (also here $ \dim Z_{ij}(1) / Z'_{ij}(1)$ is trivial independently of the complex number $t_{ij}$).
\end{description}

The cases \textbf{(b)} and \textbf{(c)} are dual of each other: the Cartier dual of the 1-motive $M_{ij}$ described in \textbf{(b)} is the 1-motive 
$M_{ij}$ described in \textbf{(c)} and viceversa.

%-------------------------------------------------

\section{Proof of the main theorem}

Consider the 1-motive $
M=[u:\ZZ \rightarrow G^n ], u(1) =(R_1, \dots, R_n)$ \eqref{GPC-1-motive} defined by the points $p_i, q_j$ and $t_{ij}$ \eqref{Points}.

\begin{notation}\label{Notation} Throughout this section, we use the following notation.

	(1) Let 
	$$p'_1, \dots ,p'_{n'}, q'_1,\dots ,q'_{r'} $$
	be a $k$--basis of the sub $k $--vector space $ \langle p_i,q_j \rangle_{i,j}$ of  $\CC/(\Omega\otimes_\ZZ\QQ)$ generated by the classes of $p_1, \dots, p_n,q_1, \dots, q_r .$ For ease of notation, we assume without loss of generality that  $p_i'=p_i$ for $1\leqslant i\leqslant n'$,  $q_j'=q_j$ for $1\leqslant j\leqslant r'.$ Remark that for $i \leqslant n'$ and $j \leqslant r'$,  $(i,j) \in \mathrm{LB}$ by \cite[Corollary 4.5]{BP}.
	
	(2) Set $u:=\dim_\QQ \langle \beta_{i,j}+\beta_{i,j}^t\rangle_{  (i,j) \in \mathrm{LB}}$ and let 
	$$\gamma_1, \dots ,\gamma_{u}$$
	be a $\QQ$--basis of the sub $\QQ $--vector space $ \langle \beta_{i,j}+\beta_{i,j}^t \rangle_{  (i,j) \in \mathrm{LB}}$ of $\mathrm{Hom}_\QQ(B,B^*) $ generated by the homomorphisms $\beta_{i,j}+\beta_{i,j}^t$ with $(i,j) \in \mathrm{LB}.$  Without loss of generality, for $m=1, \dots,u,$ we assume $\gamma_m =\beta_{i_m,j_m}+\beta_{i_m,j_m}^t$ for some $(i_m,j_m) \in \mathrm{LB}.$ 
	
	In \cite[Lemma 6.4]{Bsubmitted} we showed that $u \geqslant n' r'$ and so 
	for $m=1, \dots,n'r',$ we may assume $ i_m =1, \dots, n'$ and  $ j_m =1, \dots, r'.$

	(3) Set $v:=  \dim_{\QQ}
	\langle \textstyle{\sum_{(i,j)\in \mathrm{LB}}}	\alpha_{ij}\log( s_{ij}) \rangle$ and let
	\[
	\theta_1,\ldots,\theta_v
	\]
	be a $\QQ$--basis of the sub $\QQ$--vector space of
	$\CC/2\pi \ii\QQ$ generated by the classes of \\
	$\sum_{(i,j)\in \mathrm{LB}}
	\alpha_{ij}\log(s_{ij}),$
	with $\sum_{(i,j)\in \mathrm{LB}}
	\alpha_{ij}x_i\otimes y_j^\vee
	\in \ker (f) $
	and with $(s_{ij})$ any point of $\GG_m^{nr}$ projecting onto $
	\pi\big(pr_*\big(\psi(x_i,y_j^\vee)\big)_{(i,j)\in \mathrm{LieBracket}}\big).$
	We choose this basis among the classes of the logarithms
	of the coordinates of the trivialization point
	$ \pi(pr_*\widetilde R).$
	Thus, after reindexing if necessary, for $a=1, \dots,v,$ we assume
	\[
	\theta_a=t_{i_aj_a}
	\]
	for suitable pairs $(i_a,j_a)\in \mathrm{LB} \setminus \{ (i_m,j_m)\}_{m=1, \dots,u}.$
	The pairs $(i_a,j_a)$ are chosen so that they do
	not contribute new dimensions to $Z'(1)$ and the classes of
	$t_{i_aj_a}$ form a basis of the $\mathrm{LB}$-contribution
	to $Z(1)/Z'(1)$.
	
	(4) Set $s:=\dim_\QQ \langle t_{ij}\rangle_{ (i,j) \in \mathrm{NoLB}}$ and let 
	$$t_1, \dots ,t_{s}$$
	be a $\QQ$--basis of the sub $\QQ $--vector space $ \langle t_{ij} \rangle_{(i,j) \in \mathrm{NoLB}}$ of $\CC / 2\pi\ii \QQ$ generated by the classes of $t_{ij} $ with $(i,j) \in \mathrm{NoLB}$.  Without loss of generality, for $l=1, \dots,s,$ we assume $t_l= t_{i_lj_l}$ for some $(i_l,j_l) \in \mathrm{NoLB}.$
	
\end{notation}
\medspace

The distinguished points introduced in the above Notation play different
geometric roles.
The points $R_{i_mj_m} $ correspond to the generators $\beta_{i_mj_m}+\beta^t_{i_mj_m}$
of the torus \(Z'(1)\). They determine the
non-split fibres $\cP_{P_{i_m}, Q_{j_m}}$ from which every other LB-fibre is obtained,
modulo isogeny, and therefore account for the whole torus \(Z'(1)\). They should be regarded as representatives of the fundamental non-split fibres $\cP_{P_{i_m}, Q_{j_m}}$ underlying the 1-motive $M.$

The remaining distinguished LB-points
$ R_{i_aj_a}$
do not contribute further to \(Z'(1)\). Instead, by the
biadditivity of the Poincaré biextension, every such point decomposes,
modulo isogeny, into points lying over the fundamental non-split fibres $\cP_{P_{i_m}, Q_{j_m}}$ determined
by the points \(R_{i_mj_m}\). The corresponding differences of
trivializations on these fibres $\cP_{P_{i_m}, Q_{j_m}},$
measured modulo \(2\pi \rmi \QQ\), generate the
LB-contribution to the quotient torus \(Z(1)/Z'(1)\).

Finally, for the pairs \((i_l,j_l)\in\mathrm{NoLB}\), the fibres are split,
$ \cP_{P_{i_l}, Q_{j_l}} \simeq
(G_{j_l})_{P_{i_l}}
\simeq
\mathcal E\times\mathcal E^*\times\mathbb G_m,$
so that
$R_{i_lj_l}=(P_{i_l},Q_{j_l},e^{t_{i_lj_l}}).$
Hence the NoLB-contribution to the quotient torus
\(Z(1)/Z'(1)\) is completely determined by the toric coordinates
\(e^{t_{i_lj_l}}\), or equivalently by the classes of the logarithms
\(t_{i_lj_l}\) modulo \(2\pi \rmi \QQ\).

\medspace

We first prove that the points
$R_{i_mj_m},R_{i_aj_a}$ and $ \rme^{t_{i_lj_l}}$
which account for the dimensions of the pure motives underlying
\(\Lie\UR(M)\) are multiplicatively independent.
Next, we verify that every point \(R_{ij}\) \eqref{R} defining $M$ belongs to the multiplicative
subgroup of $G(\CC)$ generated by these points. Consequently, $R_{i_mj_m},R_{i_aj_a}$ and $ \rme^{t_{i_lj_l}}$ form a basis of
the multiplicative subgroup of \(G(\CC)\) generated by all the points
\(R_{ij}\).

\begin{proposition}\label{dimZ'(1)}
			Let $M=[u:\ZZ \rightarrow  G^n],
		u(1)=(R_1, \dots, R_n ) \in G^n(\CC),$ be the 1-motive \eqref{GPC-1-motive} defined by the complex numbers $q_j,p_i,t_{ij}$ \eqref{Points}.
		The points $R_{i_mj_m} \in G(\CC), m=1, \dots,u, $ are multiplicatively independent.
		
		In particular,
	\[ \dim Z'(1) =	\mathrm{rank} \langle  R_{i_mj_m}\rangle_{ m=1,\dots,u }.\]
\end{proposition}

\begin{proof} If all $p_i$ and $q_j$ are torsion, then $u=0.$ Hence we may assume that not all $p_i$ and $q_j$ are torsion. We distinguish two cases:
	
		1) $ \dim_k \langle p_i,q_j\rangle_{i,j} =1.$
		
	Because of Cartier duality for 1-motives, without loss of generality, we can choose $q_1$ as a $k$--basis of the $k $--vector subspace $\langle p_i,q_j\rangle_{i,j}$ of  $\CC/(\Omega\otimes_\ZZ\QQ)$ generated by the classes of $p_1, \dots, p_n,q_1, \dots, q_r.$ After possibly reindexing, we may assume that $j_1=1.$
	
	Consider the 1-motives
	\[ \qquad M_h:= \oplus_{\nu=1}^h M_{i_\nu j_\nu} \]
	for $h=1,\dots,u.$
	Let $B_h, Z'_{h}(1)$ and $Z_h(1)$ be the pure motives underlying $\Lie \UR (M_h)$ for $h=1,\dots,u.$ Since $\gamma_1, \dots ,\gamma_{u}$
	is a basis of the sub $\QQ $--vector space $ \langle \beta_{i,j}+\beta_{i,j}^t \rangle_{  (i,j) \in \mathrm{LB}},$ $Z'_{u}(1) = Z'(1) .$
	
	Consider the following multiplicative subgroups of $G(\CC)$
	\[ 
	\Gamma_h:=\langle R_{ij}\rangle_{(i,j)\in \{(i_1,j_1),\dots,(i_h,j_h)\}}
	\subseteq G(\mathbb{C})	\]
	for $h=1, \dots, u.$ Then $\Gamma_u= \langle  R_{i_mj_m}\rangle_{ m=1,\dots,u } $.
	
	For $h=1$, \(\gamma_1\) was chosen as the first element
	of a \(\QQ\)-basis of $ \langle \beta_{i,j}+\beta_{i,j}^t \rangle_{  (i,j) \in \mathrm{LB}}.$ Hence 
	\[
	\dim Z'_1(1)=1=\operatorname{rank}(\Gamma_1).
	\]
	We now prove that for each $h=2,\dots,u$,
	\begin{equation}\label{eq:2}
		\dim Z'_{h}(1)-\dim Z'_{h-1}(1)
		=
		\mathrm{rank}(\Gamma_h/\Gamma_{h-1}) =1.
	\end{equation}

	If $\mathrm{rank}(\Gamma_h/\Gamma_{h-1})=0$, then $R_{i_h j_h}$ is multiplicatively dependent
	on $\Gamma_{h-1}$, hence the $1$-motive $M_{i_h j_h}$ belongs to the tannakian category generated by $M_{h-1}.$
	Therefore $M_h$ and $M_{h-1}$ generate the same tannakian
	category, and
	\[
	\dim Z'_{h}(1)=\dim Z'_{h-1}(1).
	\]
This contradicts the choice of  $
	\gamma_h\notin 	\left\langle	\gamma_1,\ldots,\gamma_{h-1} \right\rangle_{\QQ}.$  Hence
	$
	\operatorname{rank}(\Gamma_h/\Gamma_{h-1})\neq0. $

	If $\mathrm{rank}(\Gamma_h/\Gamma_{h-1})=1$, then $R_{i_h j_h}$ is multiplicatively independent
	modulo $\Gamma_{h-1}$. In this case, the 1-motive $M_{h-1}$ is a quotient of $M_h$, and so we have an inclusion of motivic Galois groups $\Galmot(M_{h-1}) \hookrightarrow  \Galmot(M_h).$ 
Since the equality $\dim Z'_{h}(1)=\dim Z'_{h-1}(1)$ is impossible with our hypothesis on the $\gamma_h$, 
	\[
	\dim Z'_{h}(1)=\dim Z'_{h-1}(1)+1.
	\]
	This proves \eqref{eq:2}. 
	
	Summing over $h=2,\dots,u$ the equalities \eqref{eq:2}, and using the case $h=1$,  we obtain
	\[
 \dim Z'(1)=	\dim Z'_u(1)
	=
	\mathrm{rank}(\Gamma_1) +	\sum_{h=2}^u \mathrm{rank}(\Gamma_h/\Gamma_{h-1})
	=
	\mathrm{rank}(\Gamma_u) =u.
	\]
	
	2) $ \dim_k \langle p_i,q_j\rangle_{i,j}\geqslant 2.$ 
	
	 One may choose a $k$--basis $p'_1,\dots,p'_{n'},q'_1,\dots,q'_{r'}$ of the $k $--vector subspace $\langle p_i,q_j\rangle_{i,j}$ of  $\CC/(\Omega\otimes_\ZZ\QQ)$ generated by the classes of $p_1, \dots, p_n,q_1, \dots, q_r$ with $n' \geqslant 1$ and $r'\geqslant 1.$ Without loss of generality, we assume that  $p_i'=p_i$ for $i=1, \dots, n'$ and  $q_j'=q_j$ for $j=1, \dots, r'. $ By \cite[Corollary 4.5]{BP}, for $i \leqslant n'$ and $j \leqslant r'$ the couple $(i,j) \in \mathrm{LB}$.
Denote by LI the subset of LB consisting of all couples $(i,j)$ with $i=1, \dots,n'$ and $j=1, \dots,r'.$
Set $ M^{\mathrm{LI}} =\oplus_{(i,j) \in \mathrm{LI}} M_{ij}$
We add the index $ \mathrm{LI}$ to the tori underlying the unipotent radical of the 1-motives  $ M^\mathrm{LI}.$  Because of the inclusion LI $\subseteq$ LB, $Z'_\mathrm{LI}(1) $ is a subtorus of $Z'(1) $ and 
\[Z'(1) = Z'_\mathrm{LI}(1) \times Z'(1) / Z'_\mathrm{LI}(1).\]

Since the elliptic logarithms $p_1,\dots,p_{n'},q_1,\dots,q_{r'}$ are $k$--linearly independent, by \cite[Corollary 4.6]{BP} and Lemma \ref{RankForLI}
\[ \dim Z_\mathrm{LI}(1) = \dim Z'_\mathrm{LI}(1) =  \operatorname{rank}\langle R_{ij} \rangle_{(i,j)\in \mathrm{LI}} =n'r'. \]
Therefore it remains to prove that
	\begin{equation}
		\label{eq:main}
		\dim\bigl(Z'(1)/Z'_{\mathrm{LI}}(1)\bigr)
		=
		\operatorname{rank}\bigl(
		\langle R_{i_mj_m} \rangle_{m=1, \dots,u} \big/
		\langle R_{ij} \rangle_{(i,j)\in \mathrm{LI}}
		\bigr) =u-n'r'.
	\end{equation}

	Consider the 1-motives
	\[
	M_{n'r'} := M^{\mathrm{LI}}
	\quad \text{and} \quad
	M_h := M_{n'r'} \oplus ( \oplus_{m=n'r'+1}^h M_{i_m j_m}) \quad \text{for} \; h = n'r'+1,\dots, u.
	\]
 Let $B_h$, $Z'_h(1)$ and $Z_h(1)$ be the pure motives underlying $\Lie \UR (M_h)$ for $h= n'r', \dots, u.$ Clearly 
$	Z'_{n'r'}(1) = Z'_{\mathrm{LI}}(1).$  Since $\gamma_1, \dots ,\gamma_{u}$
is a basis of the $\QQ $--vector space $ \langle \beta_{i,j}+\beta_{i,j}^t \rangle_{  (i,j) \in \mathrm{LB}},$ $Z'_{u}(1) = Z'(1) .$
	
	Consider the following multiplicative subgroups of $G(\mathbb{C})$:
	\[
	\Gamma_{n'r'} := \langle R_{ij} \rangle_{(i,j)\in \mathrm{LI}}
	\quad \text{and} \quad
	\Gamma_h := \langle R_{ij} \rangle_{(i,j)\in \mathrm{LI} \cup \{(i_{n'r'+1},j_{n'r'+1}),\dots,(i_h,j_h)\}} \quad \text{for} \; h = n'r'+1,\dots,u.
	\]
 Then $\Gamma_{u} = \langle R_{i_mj_m} \rangle_{m=1, \dots,u}$.
	
	We now prove that for each $h=n'r'+1,\dots,u$,
	\begin{equation}
		\label{eq:increment}
		\dim Z'_h(1) - \dim Z'_{h-1}(1)
		=
		\operatorname{rank}(\Gamma_h / \Gamma_{h-1}) =1.
	\end{equation}

	If $\operatorname{rank}(\Gamma_h / \Gamma_{h-1}) = 0$, then $R_{i_h j_h}$ is multiplicatively dependent on $\Gamma_{h-1}$, and so the 1-motive $M_{i_h j_h}$ belongs to the tannakian category generated by $M_{h-1}$. Therefore $M_h$ and $M_{h-1}$ generate the same tannakian category, and
	\[
	\dim Z'_h(1) = \dim Z'_{h-1}(1).
	\]
	This contradicts the choice of $
	\gamma_h\notin 	\left\langle	\gamma_1,\ldots,\gamma_{h-1} \right\rangle_{\QQ}.$ Hence
	$
	\operatorname{rank}(\Gamma_h/\Gamma_{h-1})\neq0. $
	
	If $\operatorname{rank}(\Gamma_h / \Gamma_{h-1}) = 1$, then $R_{i_h j_h}$ is multiplicatively independent modulo $\Gamma_{h-1}$. In this case, the 1-motive $M_{h-1}$ is a quotient of $M_h$, and so we have an inclusion of motivic Galois groups $\Galmot(M_{h-1}) \hookrightarrow  \Galmot(M_h).$ 
	Since the equality $\dim Z'_{h}(1)=\dim Z'_{h-1}(1)$ is impossible with our hypothesis on the $\gamma_h$, 
	\[
	\dim Z'_{h}(1)=\dim Z'_{h-1}(1)+1.
	\]
	This proves \eqref{eq:increment}. 
	
	Summing over $h=n'r'+1,\dots,u$ the equalities \eqref{eq:increment}, we obtain
	\[
	\dim Z'(1) - \dim Z'_{\mathrm{LI}}(1)
	=
	\sum_{h=n'r'+1}^u \operatorname{rank}(\Gamma_h / \Gamma_{h-1})
	=
	\operatorname{rank}(\Gamma_u / \Gamma_{n'r'}) =u-n'r',
	\]
	which furnishes \eqref{eq:main}.
\end{proof}

The following remark complements Remark~\ref{Contribution} and should be read together with it. Remark~\ref{Contribution}  explains the decomposition of the dimension of the quotient torus $Z(1)/Z'(1),$ whereas the following remark describes the decomposition of the multiplicative rank associated with the LB-points.

\begin{remark}\label{Contribution2}
	The two families of points
	$	\{R_{i_mj_m}\}_{m=1,\ldots,u} $ and $
	\{R_{i_aj_a}\}_{a=1,\ldots,v}	$
	are multiplicatively independent.
	Indeed, the first family accounts for
	the torus \(Z'(1)\), whereas the second accounts for the
	LB-contribution to the quotient torus \(Z(1)/Z'(1)\). Because of the short exact sequence
	$	0\rightarrow Z'(1)\rightarrow Z(1)
	\rightarrow Z(1)/Z'(1)\rightarrow 0,$
	the two families contribute to different factors of the torus $Z(1)$ and
	their multiplicative contributions are therefore independent. Equivalently,
	$$	\mathrm{rank} \langle  R_{i_mj_m}\rangle_{ m=1,\dots,u} +	\mathrm{rank} \langle  R_{i_aj_a}\rangle_{ a=1,\dots,v} = \mathrm{rank} \langle  R_{i_mj_m},  R_{i_aj_a} \rangle_{ \stackrel{m=1,\dots,u}{a=1,\dots,v}}.$$
\end{remark}

\begin{proposition}\label{dimZ(1)/Z'(1)}
	Let $M=[u:\ZZ \rightarrow  G^n],
	u(1)=(R_1, \dots, R_n ) \in G^n(\CC),$ be the 1-motive \eqref{GPC-1-motive} defined by the complex numbers $q_j,p_i,t_{ij}$ \eqref{Points}.
	The points $ R_{i_aj_a} \in G(\CC), a=1, \dots,v ,$ are multiplicatively independent. Moreover the points 
	$\rme^{t_{i_lj_l}} \in \GG_m(\CC), l=1, \dots,s , $ are also multiplicatively independent.
	
	In particular,
	\[ \dim Z(1)/Z'(1) =	\mathrm{rank} \langle  R_{i_aj_a}\rangle_{ a=1,\dots,v} + \mathrm{rank}  \langle \rme^{t_{i_lj_l}}\rangle_{ l=1,\dots,s}.\]
\end{proposition}

\begin{proof} The classes of the $t_{i_lj_l}$ form a basis of the sub $\QQ$--vector space $ \langle t_{ij}\rangle_{ (i,j) \in \mathrm{NoLB}}$ of $\CC / 2\pi\ii \QQ,$ which is the NoLB-contribution to the quotient torus $Z(1)/Z'(1).$ Hence their classes are $\QQ$--linearly independent modulo $2 \pi \rmi \QQ$ and the rank of the multiplicative group $ \langle \rme^{t_{i_lj_l}}\rangle_{ l=1,\dots,s}$ is $s.$ 
	
	  Since the $\gamma_1, \dots, \gamma_u$ form a basis of $Z'(1),$
each \(\beta_{i_aj_a}+\beta_{i_aj_a}^t\) is a $\QQ$--linear combination of the
\(\gamma_m\)'s. Thus, modulo isogeny, \(R_{i_aj_a}\) lies over a fibre whose
LB-contribution has already been accounted for by the chosen
points \(R_{i_mj_m}\). Hence the point $R_{i_aj_a}$ can only contribute to the
quotient $Z(1)/Z'(1).$

	Consider the 1-motives
\[
M_{0} := \oplus_{m=1}^u M_{j_mj_m}
\quad \text{and} \quad
M_h := M_0 \oplus (\oplus_{a =1}^h M_{i_a j_a}) \quad \text{for} \; h = 1,\dots, v.
\]
Let $B_h$, $Z'_h(1)$ and $Z_h(1)$ be the pure motives underlying $\Lie \UR (M_h)$  for $h= 0, \dots, u.$ Clearly 
$	Z'_{h}(1) = Z'(1)$  for $h=0, \dots,v$ and $\dim Z_0(1)/Z_0'(1)=0.$ Since the classes of the complex numbers $t_{i_aj_a},$ for $a=1, \dots, v$, are a 
basis of the $\QQ$--vector space
$\sum_{(i,j)\in \mathrm{LB}}
\alpha_{ij}\log(s_{ij}),$  the dimension of $Z_v(1)/Z_v'(1)$ is the LB-contribution to the quotient torus $Z(1)/Z'(1).$

Consider the following multiplicative subgroups of $G(\mathbb{C})$:
\[
\Gamma_{0} := \langle R_{i_mj_m} \rangle_{m=1, \dots,u}
\quad \text{and} \quad
\Gamma_h := \langle R_{i_mj_m}, R_{ij} \rangle_{\stackrel{ m=1, \dots,u}{(i,j) \in \{(i_1,j_1),\dots,(i_h,j_h)\}}}  \quad \text{for} \; h = 1,\dots,v.
\]
Note that $ \mathrm{rank} (\Gamma_v /\Gamma_0 )=  \mathrm{rank}\langle  R_{i_aj_a}\rangle_{ a=1,\dots,v}$ by Remark \ref{Contribution2}.

For $h=1$, the class of \(t_{i_1j_1}\) was chosen as the first element
of a \(\QQ\)-basis of the LB-contribution to
\(Z(1)/Z'(1)\). Hence, recalling that $\dim Z_0(1)/Z_0'(1)=0 $ and Remark \ref{Contribution2},
we have
\[
\dim Z_1(1)/Z'_1(1) =1=\operatorname{rank}(\Gamma_1/\Gamma_0).
\]
We now prove that for each $h=2,\dots,v$,
\begin{equation}
	\label{eq:increment2}
	\dim Z_h(1)/Z'_h(1) - \dim Z_{h-1}/ Z'_{h-1}(1)
	=
	\operatorname{rank}(\Gamma_h / \Gamma_{h-1}) =1.
\end{equation}

If $\operatorname{rank}(\Gamma_h / \Gamma_{h-1}) = 0$, then $R_{i_h j_h}$ is multiplicatively dependent on $\Gamma_{h-1}$, and so the 1-motive $M_{i_h j_h}$ belongs to the tannakian category generated by $M_{h-1}$. Therefore $M_h$ and $M_{h-1}$ generate the same tannakian category, and
\[
	\dim Z_h(1)/Z'_h(1) = \dim Z_{h-1}(1)/ Z'_{h-1}(1)
\]
This contradicts the choice of $
t_{i_hj_h} \notin 	\left\langle	t_{i_1j_1},\ldots, t_{i_{h-1}, j_{h-1}} \right\rangle_{\QQ}.$
 Hence
$
\operatorname{rank}(\Gamma_h/\Gamma_{h-1})\neq0. $

If $\operatorname{rank}(\Gamma_h / \Gamma_{h-1}) = 1$, then $R_{i_h j_h}$ is multiplicatively independent modulo $\Gamma_{h-1}$. In this case, the 1-motive $M_{h-1}$ is a quotient of $M_h$, and so we have an inclusion of motivic Galois groups $\Galmot(M_{h-1}) \hookrightarrow  \Galmot(M_h).$ 
	Since the equality $	\dim Z_h(1)/Z'_h(1) = \dim Z_{h-1} (1)/ Z'_{h-1}(1)$ is impossible with our hypothesis on the $t_{i_aj_a}$, 
\[
\dim Z_h(1)/Z'_h(1) = \dim Z_{h-1}(1)/ Z'_{h-1}(1) +1.
\]
This proves \eqref{eq:increment2}.

	Summing over $h=2,\dots,v$ the equalities \eqref{eq:increment2}, and using the case $h=1$,  we obtain
\[
\dim Z(1)/Z'(1) - \mathrm{rank}  \langle \rme^{t_{i_lj_l}}\rangle_{ l=1,\dots,s} = \dim Z_v(1)/Z_v'(1)  
=
\operatorname{rank}(\Gamma_v / \Gamma_0 ) =v.
\]
\end{proof}

Now we verify that every point \(R_{ij}\) \eqref{R} defining $M$ belongs to the multiplicative
subgroup of $G(\CC)$ generated by the points
$R_{i_mj_m},R_{i_aj_a}$ and $ \rme^{t_{i_lj_l}}.$

\begin{theorem}\label{teo:GeneralPoints}
		Let $M=[u:\ZZ \rightarrow  G^n],
	u(1)=(R_1, \dots, R_n ) \in G^n(\CC),$ be the 1-motive \eqref{GPC-1-motive} defined by the complex numbers $q_j,p_i,t_{ij}$ \eqref{Points}. We have the equalities
	\begin{align*}
	 \langle  R_{ij}\rangle_{  (i,j) \in \mathrm{LB} }&= \langle  R_{i_mj_m},  R_{i_aj_a}\rangle_{ m=1,\dots,u \atop a=1,\dots,v} ,\\
	  \langle  R_{ij}\rangle_{  (i,j) \in \mathrm{NoLB} }& \cong  \langle \rme^{t_{i_lj_l}}\rangle_{ l=1,\dots,s},\\
	  \langle  R_{ij} \rangle_{ i=1, \dots, n \atop j= 1, \dots, r }&=   \langle  R_{ij}\rangle_{  (i,j) \in \mathrm{LB} } \cdot  \langle  R_{ij}\rangle_{  (i,j) \in \mathrm{NoLB} }.
	\end{align*}
	Moreover,
	\begin{align*}
		\mathrm{rank}  \langle  R_{ij}\rangle_{  (i,j) \in \mathrm{LB} }&= 	\mathrm{rank} \langle  R_{i_mj_m}\rangle_{ m=1,\dots,u} +	\mathrm{rank} \langle  R_{i_aj_a}\rangle_{ a=1,\dots,v} ,\\
		\mathrm{rank}  \langle  R_{ij}\rangle_{  (i,j) \in \mathrm{NoLB} }&=\mathrm{rank}  \langle \rme^{t_{i_lj_l}}\rangle_{ l=1,\dots,s},\\
		\mathrm{rank}  \langle  R_{ij}\rangle_{ i=1, \dots, n \atop j= 1, \dots, r }&= 	\mathrm{rank}  \langle  R_{ij}\rangle_{  (i,j) \in \mathrm{LB} } + \mathrm{rank}  \langle  R_{ij}\rangle_{  (i,j) \in \mathrm{NoLB} }.
	\end{align*}
\end{theorem}

\begin{proof} 
	
	For every pair $
	(i,j)\in
	\mathrm{LB}\setminus
	(	\{(i_m,j_m)\}_{m=1,\ldots,u}
	\cup
	\{(i_a,j_a)\}_{a=1,\ldots,v}),
	$
	we show that
	\[
	R_{ij}
	\text{ is multiplicatively dependent on }
	\langle
	R_{i_mj_m},
	R_{i_aj_a}
	\rangle_{m,a}.
	\]
	The multiplicative independence of the distinguished points
	\(R_{i_mj_m}\) and \(R_{i_aj_a}\) has already been established in
	Propositions~\ref{dimZ'(1)} and~\ref{dimZ(1)/Z'(1)}. Fix a pair $(i,j)\in
	\mathrm{LB}\setminus
	(	\{(i_m,j_m)\}_{m=1,\ldots,u}
	\cup
	\{(i_a,j_a)\}_{a=1,\ldots,v}),$
 and consider the 1-motives
\[
M_{0} := (\oplus_{m=1}^u M_{i_mj_m}) \oplus (\oplus_{a=1}^v M_{j_aj_a})
\qquad \text{and} \qquad
M_1 := M_0 \oplus M_{ij} 
\]
Let $B_h$, $Z'_h(1)$ and $Z_h(1)$ be the pure motives underlying $\Lie \UR (M_h)$  for $h= 0,1.$
By Notation~\ref{Notation}(2), we have chosen the pairs
$(i_m,j_m), m=1,\ldots,n'r',$
so that $i_m=1,\ldots,n'$ and $j_m=1,\ldots,r'$
where $ p_1,\ldots,p_{n'},q_1,\ldots,q_{r'} $
form a \(k\)-basis of the sub \(k\)-vector space
$\langle p_i,q_j\rangle_{i,j}
\subset
\CC/(\Omega\otimes_{\ZZ}\QQ)
$
generated by the classes of
\(p_1,\ldots,p_n,q_1,\ldots,q_r\).
Therefore 
$B_h=B$ for any $h.$
 Since $\gamma_1, \dots ,\gamma_{u}$
is a basis of the $\QQ $--vector space $ \langle \beta_{i,j}+\beta_{i,j}^t \rangle_{  (i,j) \in \mathrm{LB}},$ $Z'_{0}(1) = Z'(1) $ and $\dim Z'_{0}(1)=u.$
Moreover the classes of the complex numbers $t_{i_aj_a},$ $a=1, \dots, v$, are a 
basis of the $\QQ$--vector space
$\sum_{(i,j)\in \mathrm{LB}}
\alpha_{ij}\log(s_{ij}).$ Hence the dimension of $Z_0(1)/Z_0'(1)$ is the LB-contribution to the quotient torus $Z(1)/Z'(1)$ and $\dim Z_0(1)/Z_0'(1)=v.$

Consider the following multiplicative subgroups of $G(\mathbb{C})$:
\[
\Gamma_{0} := \langle R_{i_mj_m}, R_{i_aj_a} \rangle_{m,a}
\qquad \text{and} \qquad
\Gamma_1 := \langle R_{i_mj_m}, R_{i_aj_a}, R_{ij} \rangle_{ m,a}.
\]

Assume by contradiction that $
\operatorname{rank}(\Gamma_1/\Gamma_0)=1.$
Then \(R_{ij}\) is multiplicatively independent modulo \(\Gamma_0\).
The 1-motive $M_{0}$ is then a quotient of $M_1$, and so we have an inclusion of motivic Galois groups $\Galmot(M_{0}) \hookrightarrow  \Galmot(M_1).$ 
Since \(B_h=B\) for \(h=0,1\), this independence can only give a new toric contribution either to \(Z'_1(1)\) or to the LB-contribution to
\(Z_1(1)/Z'_1(1)\).
The first possibility would imply
\[
\dim Z'_1(1)=\dim Z'_0(1)+1=u+1,
\]
which is impossible because
$
\beta_{ij}+\beta_{ij}^t
\in
\left\langle
\beta_{i_mj_m}+\beta_{i_mj_m}^t
\right\rangle_{m=1,\ldots,u}.
$
The second possibility would imply
\[
\dim Z_1(1)/Z'_1(1)
=
\dim Z_0(1)/Z'_0(1)+1
=
v+1,
\]
which is impossible because the class of \(t_{ij}\) belongs to the
\(\QQ\)-vector space generated by the classes of
\(t_{i_aj_a}\), \(a=1,\ldots,v\).
Therefore $	\operatorname{rank}(\Gamma_1/\Gamma_0)=0,$
and \(R_{ij}\) is multiplicatively dependent on \(\Gamma_0\).
Repeating the above argument for every pair
$	(i,j)\in
\mathrm{LB}\setminus
	(	\{(i_m,j_m)\}_{m=1,\ldots,u}
\cup
\{(i_a,j_a)\}_{a=1,\ldots,v}),$
we conclude that every remaining point \(R_{ij}\) is multiplicatively
dependent on the subgroup
$
\left\langle R_{i_mj_m}, R_{i_aj_a} \right\rangle_{m,a}.$ Recalling Remark \ref{Contribution2}
 we get the theorem for the LB-contribution.

	We now consider the pairs $(i,j)$ belonging to NoLB. Since the abelian part is already fixed, the points
	\(R_{i_lj_l}\) differ only by their toric
	coordinates. Hence the multiplicative subgroup of $G(\CC)$ generated by the points
	\(R_{i_lj_l}\) is isomorphic to the multiplicative subgroup of
	\(\GG_m(\CC)\) generated by the points \(e^{t_{i_lj_l}}\).
	Consequently,
	\[
	\rank\langle R_{i_lj_l}\rangle_{l=1,\ldots,s}
	=
	\rank\langle e^{t_{i_lj_l}}\rangle_{l=1,\ldots,s}.
	\]
	For every pair $(i,j)\in
	\mathrm{NoLB}\setminus
	\{(i_l,j_l)\}_{l=1,\ldots,s}$ 
	we prove that
	\[
	R_{ij}
	\text{ is multiplicatively dependent on }
	\langle
	R_{i_lj_l}
	\rangle_{l=1,\ldots,s}.
	\]
The multiplicative independence of the distinguished points
\(R_{i_lj_l}\) has already been established in Proposition \ref{dimZ(1)/Z'(1)}.
				Fix a pair $(i,j)\in
		\mathrm{NoLB}\setminus
		\{(i_l,j_l)\}_{l=1,\ldots,s}$ and consider the 1-motives
		\[
		M_{0} := [0 \to B] \oplus (\oplus_{l=1}^s M_{j_lj_l})
		\qquad \text{and} \qquad
		M_1 := M_0 \oplus M_{ij} 
		\]
		Let $B_h$, $Z'_h(1)$ and $Z_h(1)$ be the pure motives underlying $\Lie \UR (M_h)$  for $h= 0,1.$ Clearly 
		$ B_h=B	.$  Since \((i_l,j_l)\in\mathrm{NoLB}\) and \((i,j)\in\mathrm{NoLB}\),
		we have $ \dim	Z'_0(1)= \dim Z'_1(1)=0.$
		Moreover, the classes of \(t_{i_lj_l}\), \(l=1,\ldots,s\), form a
		\(\QQ\)-basis of $
		\langle t_{ij}\rangle_{(i,j)\in\mathrm{NoLB}}
		\subset \CC/2\pi\ii\QQ.$
		Hence $
		\dim Z_0(1)=\dim Z_1(1)=s.$
		
		Consider the following multiplicative subgroups of $G(\mathbb{C})$:
		\[
		\Gamma_{0} := \langle R_{i_lj_l} \rangle_{l=1, \dots,s}
		\qquad \text{and} \qquad
		\Gamma_1 := \langle R_{i_lj_l}, R_{ij} \rangle_{ l=1, \dots,s}.
		\]
		
		Assume by contradiction that $
		\operatorname{rank}(\Gamma_1/\Gamma_0)=1.$
		Then \(R_{ij}\) is multiplicatively independent modulo \(\Gamma_0\).
		The 1-motive $M_{0}$ is then a quotient of $M_1$, and so we have an inclusion of motivic Galois groups $\Galmot(M_{0}) \hookrightarrow  \Galmot(M_1).$ 
		Since the abelian part is fixed and the Lie bracket vanishes for
		NoLB-pairs, this independence can only produce a new toric contribution to
		\(Z_1(1)/Z'_1(1)\). Thus
		\[
		\dim Z_1(1)/Z'_1(1)
		=
		\dim Z_0(1)/Z'_0(1)+1
		=
		s+1.
		\]
		This contradicts the fact that the class of \(t_{ij}\) belongs to
		$	\left\langle t_{i_lj_l}\right\rangle_{l=1,\ldots,s}.$
		Therefore $	\operatorname{rank}(\Gamma_1/\Gamma_0)=0,$
		and \(R_{ij}\) is multiplicatively dependent on \(\Gamma_0\).
	Repeating the above argument for every pair
		$	(i,j)\in
		\mathrm{NoLB}\setminus
		\{(i_l,j_l)\}_{l=1,\ldots,s},$
		we conclude that every remaining point \(R_{ij}\) is multiplicatively
		dependent on the subgroup
		$
		\left\langle R_{i_lj_l}\right\rangle_{l=1,\ldots,s},$
		which proves the theorem for the NoLB-contribution.

		Finally, the sets LB and NoLB form a partition of $\{1, \dots,n\} \times \{1 ,\dots ,r\}.$ Hence every pair $(i,j)$ belongs to exactly one of them. Moreover, every LB-point is generated by the distinguished LB-points $R_{i_mj_m},R_{i_aj_a},$ and every NoLB-point is generated by the distinguished NoLB-points $R_{i_lj_l}.$ These distinguished generators are mutually independent by Remarks \ref{Contribution}, \ref{Contribution2}.  This proves the last equality of the theorem and,
		passing to ranks, the last equality concerning the multiplicative
		ranks.
\end{proof}

\begin{proof}[Proof of Theorem \ref{Teo:dimGal(M)}] Without loss of generality we may assume the field of definition $K$ \eqref{FieldDefinition} of the 1-motive $M$ to be algebraically closed. By \cite[Theorem 1.2.1]{A19}
	the motivic Galois group of $M$ and its Mumford-Tate group coincide. Hence \cite[Lemma 3.5]{BPSS} implies that $\dim \UR(M)= 2 \dim B + \dim Z(1).$ Recalling that the dimension of the motivic Galois group of an elliptic curve $\cE$ is $\frac{4}{\dim_\QQ k},$ from the short exact sequence \eqref{eq:shortexactsequenceUR},
Proposition \ref{dimZ'(1)}, Proposition \ref{dimZ(1)/Z'(1)} and Theorem \ref{teo:GeneralPoints}, we get
\[\dim \Galmot(M) \]
\begin{align*}
	&= \frac{4}{\dim_{\QQ}k}
	+
	2\dim_k\langle p_i,q_j\rangle_{i=1, \dots,n \atop j=1, \dots, r}
	+
\overbrace{\rank\langle R_{i_mj_m}\rangle_{m=1,\dots,u}}^{\dim Z'(1)}
+
\overbrace{
	\rank\langle R_{i_aj_a}\rangle_{a=1,\dots,v}
	+
	\rank\langle e^{t_{i_lj_l}}\rangle_{l=1,\dots,s}
}^{\dim Z(1)/Z'(1)}
\\
&= \frac{4}{\dim_{\QQ}k}
+
2\dim_k\langle p_i,q_j\rangle_{i=1, \dots,n \atop j=1, \dots, r}
+
\underbrace{
	\rank\langle R_{i_mj_m}\rangle_{m=1,\dots,u}
	+
	\rank\langle R_{i_aj_a}\rangle_{a=1,\dots,v}
}_{\rank\langle R_{ij}\rangle_{(i,j)\in \mathrm{LB}}}
+
\underbrace{
	\rank\langle e^{t_{i_lj_l}}\rangle_{l=1,\dots,s}
}_{\rank\langle R_{ij}\rangle_{(i,j)\in \mathrm{NoLB}}},\\
&=
\frac{4}{\dim_{\QQ}k}
+
2\dim_k\langle p_i,q_j\rangle_{i=1, \dots,n \atop j=1, \dots, r}
+ 
\rank\langle R_{ij}\rangle_{i=1, \dots,n \atop j=1, \dots, r}.\\
\end{align*}
\end{proof}

The cases $n=0$ and $r=0$ are dual to each other via Cartier duality for 1-motives. If $r=0$ (equivalently, $q_j \in \Omega \otimes_\mathbb{Z}\mathbb{Q}$ for all $j$), then the set $\mathrm{LB}$ is empty, and the above Theorem reduces to \eqref{DimSplit}.

%-------------------------------------------
\bibliographystyle{plain}

\end{document}